\documentclass[11pt]{amsart}
\usepackage{amsmath,amsthm,latexsym,amssymb}
\usepackage{graphicx}
\usepackage{color}
\usepackage{epsfig}
\numberwithin{equation}{section}

\newtheorem{theorem}{Theorem}
\newtheorem{lemma}{Lemma}
\newtheorem{corollary}{Corollary}
\newtheorem{proposition}{Proposition}
\newtheorem{remark}{Remark}

\numberwithin{theorem}{section}

\numberwithin{corollary}{section}
\numberwithin{lemma}{section}
\numberwithin{definition}{section}
\numberwithin{proposition}{section}
\numberwithin{remark}{section}

\newcommand{\R}{\mathbb R}
\newcommand{\N}{\mathbb N}

\newcommand{\medint}{-\kern  -,375cm\int}

\title[Optimal Poincar\'e inequalities]{Best constants in Poincar\'e inequalities for convex domains.}
\author[L. Esposito - C. Nitsch  - C. Trombetti]{L. Esposito$^\dag$ - C. Nitsch$^*$    - C. Trombetti$^*$}
 \thanks{%
$^*$ Dipartimento di Matematica e Applicazioni ``R. Caccioppoli'', Universit\`{a}
degli Studi di Napoli Federico II, Complesso Monte S. Angelo, via Cintia
- 80126 Napoli, Italy; \\ email: c.nitsch@unina.it;
cristina@unina.it\\
\indent$^\dag$ Dipartimento di Matematica, Universit\`{a}
degli Studi di Salerno,Via Ponte Don 
Melillo, I-84084 Fisciano (SA), Italy.\\
email: luesposi@unisa.it}


\begin{document}

\maketitle

\begin{abstract}
We prove a Payne-Weinberger type inequality for the $p$-Laplacian Neumann eigenvalues ($p\ge 2$). The inequality provides the sharp upper bound on convex domains, in terms of the diameter alone, of the best constants in Poincar\'e inequality. The key point is the implementation of a refinement of the classical P\'olya-Szeg\"o inequality for the symmetric decreasing rearrangement which yields an optimal weighted Wirtinger inequality.
\end{abstract}

\section{Introduction}
Let $\Omega$ be an open bounded Lipschitz connected subset in $\R^n$. The first nontrivial Neumann eigenvalue $\mu_p$ for the $p$-Laplacian equation ($p>1$)
\begin{equation}
\left\{\begin{array}{ll}
-{\rm div}(|Du|^{p-2}Du)=\mu_p |u|^{p-2}u& \mbox{in $\Omega$}\\\\
\dfrac{\partial u}{\partial \nu}=0 & \mbox{on $\partial \Omega$}
\end{array}\right .
\end{equation}
can be characterized by

\begin{equation}\label{eq_quoz}
\mu_p=\mathop{\min_{u\in W^{1,p}(\Omega)}}_{\int_\Omega|u|^{p-2}u=0} \dfrac{\displaystyle\int_\Omega |Du|^p}{\displaystyle\int_\Omega |u|^p}
\end{equation}

It is well known that (see for instance \cite{Ma}), for any given open bounded Lipschitz connected $\Omega$, a Poincar\'e inequality holds true, in the sense that there exists a constant $C_{\Omega,p}$ such that
\begin{equation}\label{eq_poinc}
\inf_{t\in \R} \|u-t\|_{L^p(\Omega)}\le C_{\Omega,p} \|Du\|_{L^p(\Omega)},
\end{equation}
for any $u\in W^{1,p}(\Omega)$.

The value of the best constant in \eqref{eq_poinc} is
 $$C_{\Omega,p}=\mu_p^{-1/p}.$$

In \cite{pawe}, and more recently in \cite{be}, it has been proved that, if $p=2$, $\Omega$ is convex, in any dimension
\begin{equation}\label{eq_classica}
\frac{1}{C_{\Omega,2}}=\sqrt{\mu_2} \ge  \frac{\pi}{d},
\end{equation}
where $d$ is the diameter of $\Omega$.
Observe that the right hand side of \eqref{eq_classica} is exactly the value achieved,  in dimension $n=1$ on any interval of length $d$, by the first nontrivial Laplacian eigenvalue (without distinction between the Neumann and the Dirichlet problems).

The proof of \eqref{eq_classica} in \cite{pawe, be} indeed relies on the reduction to a one dimensional problem. At this aim, for any given smooth admissible test function  $u$ in the right hand side of \eqref{eq_quoz}, the authors show that it is possible to perform a clever slicing of the domain $\Omega$ in convex sets which are as tiny as desired in at least $n-1$ orthogonal directions. On each one of such convex components of $\Omega$, they are able to show that the Rayleight quotient of $u$ can be approximated by a 1-dimensional weighted Rayleight quotient. This leads the authors to look for the best constants of a class of one dimensional weighted Poincar\'e-Wirtinger inequalities.

However the technique that they use strongly relies on the linearity of the Laplace operator (in particular the authors need the property that derivatives of solutions to homogeneous differential equation are still solutions to the same equation). For this reason the proof can not be generalized to $p\ne 2$ in a straightforward way, and to our knowledge the only other case which up to now have been investigated is the limit case $p=1$ (see \cite{acdur}).
 
The main result of this paper is actually to prove that
\begin{theorem}[Main Theorem]\label{teo_main}
Let $\Omega\subset\R^n$ be a bounded convex set having diameter $d$. For $p\ge 2$ and in any dimension we have $$\mu_p \ge  \left(\frac{\pi_p}{d}\right)^p$$
where 
\begin{equation}\label{eq_pp}
\pi_p=2\int_0^{(p-1)^{1/p}}\frac{dt}{(1-t^p/(p-1))^{1/p}}=2\pi\frac{(p-1)^{1/p}}{p(\sin(\pi/p))}
\end{equation}
\end{theorem}

We observe that $\displaystyle \left(\frac{\pi_p}{d}\right)^p$ for $n=1$ is the $p$-Laplacian eigenvalue on $[0,d]$, without distinction between the Neumann and the Dirichlet boundary conditions (see for instance \cite{drab}, \cite{rewa})

The formula \eqref{eq_pp} for $\pi_p$ can be found for instance in \cite{ drab, li, rewa, sta1, sta2}. The fact that $\pi_2=\pi$ is consistent with the classical Wirtinger inequality (see \cite{HL}) and obviously also with \eqref{eq_classica}.

The result stated in Theorem \ref{teo_main} is the consequence of the following estimate on the best constant in a class of weighted Wirtinger inequalities (for an insight into weighted Wirtinger inequalities, and more generally into weighted Hardy inequalities, we refer to \cite{kp,Ma,tur}).
\begin{proposition}\label{pr_wirtinger}
Let $f$ be a nonnegative log-concave function defined on $[0,L]$ and $p\ge 2$
then
\begin{equation}\label{eq_wirtinger}
\mathop{\inf_{u\in W^{1,p}(0,L)}}_{\int_0^L f|u|^{p-2}u=0} \frac{\displaystyle \int_0^L f(x)|u'(x)|^p dx}{\displaystyle\int_0^L f(x)|u(x)|^p dx}\ge\mathop{\min_{u\in W^{1,p}(0,L)}}_{\int_0^L |u|^{p-2}u=0} \frac{\displaystyle \int_0^L |u'(x)|^p dx}{\displaystyle\int_0^L |u(x)|^p dx}=\left(\frac{\pi_p}{L}\right)^p.
\end{equation}
\end{proposition}



One of the main tools in the proof of Proposition \ref{pr_wirtinger} is a refinement of the P\'olya-Szeg\"o principle which, to our knowledge, is completely new. 
We provide such a proof in section \ref{sec_polya} after recalling the notion of symmetric decreasing rearrangement and a short list of its main properties. Thereafter we also explain, through a simple example, to which extent our result can not be deduced by the classical one. 

Next, in section \ref{sec_wirtinger}, we prove the weighted Wirtinger inequality stated in Proposition \ref{pr_wirtinger} which is obtained piling up Lemma \ref{lem_kappa}, Lemma \ref{lem_allkappa} and Lemma \ref{lem_fund}. Finally, in section \ref{sec_red}, we show that Proposition \ref{pr_wirtinger} implies Theorem \ref{teo_main}. This last proof is given for the seek of completeness but it can be considered an adaptation when $p>2$ of the proof considered when $p=2$ in \cite{pawe} and later in \cite{be} (see also \cite{acdur} for the $1$-Laplacian).

\section{A P\'olya-Szeg\"o type inequality}\label{sec_polya}

The aim of this section is the proof of a one dimensional P\'olya-Szeg\"o type principles.
Therefore, first of all, we recall the notion of symmetric decreasing rearrangement. Let $u: \R \rightarrow \R$ be a measurable function, 
denoted by $\mu_u(t)={\mathcal L}^1\{x\in\R:|u(x)|>t\}$  the distribution function of $u$,
the symmetric decreasing rearrangement of $u$ is the function 
$$u^\sharp(x)=\left[
\begin{array}{ll}
\displaystyle\sup_{\R} u& x=0\\\\
\inf\{t\ge 0:\mu_u(t)<2|x|\}&x\ne0
\end{array}\right .$$
As an immediate consequence of its definition, the symmetric decreasing rearrangement $u^\sharp$ and the function $u$ are equimeasurable, hence,
if $u\in L^p(\R)$ then $$\|u\|_p=\|u^\sharp\|_p.$$
Another useful and not so obvious property is the well known P\'olya-Szeg\"o principle, according to which if $u\in W^{1,p}(\R)$ is compactly supported, then $$\|u'\|_p\ge\|u^\sharp{'}\|_p.$$
Under the same hypothesis on the function $u$, it is possible to prove that if $$G(s,t): (s,t)\in\R^+\times\R^+\to \R$$ is a measurable nonnegative function, convex and nondecreasing in $t$, then
$$\int_\R G(|u|,|u'|)\ge\int_\R G(u^\sharp,|u^{\sharp}{'}|)$$

For more details on rearrangements see \cite{kaw}.\\

In the following we shall prove an inequality that reminds the P\'olya-Szeg\"o principle but to our knowledge is not covered by any known result in the literature.
\begin{proposition}\label{pr_polya}
Let $u$ be in $W^{1,p}(\R)$ ($p\ge 2$) a compactly supported function. Then for all real $\kappa$
$$\int_\R|u'(x)+\kappa u(x)|^pdx\ge\int_\R|u^\sharp {'}(x)|^p dx.$$
\end{proposition} 
\begin{proof}
We can always assume that $u$ is nonnegative, if not we consider its modulus. Moreover we suppose that $u$ is a simple function (i.e.: continuous and piecewise linear). The general case can be deduced by approximation arguments.

We decompose the set $[0,\max u]$ in the union of finitely many intervals $(a_i,a_{i+1})$, $i=0,...,N$ $(N\in\N)$ such that $a_i<a_{i+1}$. Denoting by $D_i=\{x\in \R: a_i < u(x)< a_{i+1}\}$, then we chose $a_i$ such that $D_i$ is the union of an even number of open intervals, namely $X_{i,j}$ (with $j=1,...,2M_j$ and $M_j\in\N$) in each of which $u$ is linear and non constant.

We also denote by $u'_{i,j}$ the constant value that the derivative of $u$ takes on $X_{i,j}$ and, without loss of generality, we can assume $u'_{i,j}=(-1)^{j+1}|u'_{i,j}|$

Taking into account the definition of symmetric rearrangement,  $u^\sharp$ is simple when $u$ is simple. Moreover,  denoting by $D^\sharp_i=\{x\in \R: a_i < u^\sharp(x)< a_{i+1}\},$ equimeasurability of $u$ and $u^\sharp$ implies 
$$ \frac{2}{|u^\sharp{'}(x)|} =-\mu '_u(t)=\sum_{j=1}^{2M_j} \frac{1}{|u'_{i,j}|} \qquad  x\in D^\sharp_i,\quad a_i<t<a_{i+1}.$$ 

Now we can use the coarea formula to get

\begin{align*}
\displaystyle \int_\R|u'(x)+\kappa u(x)|^pdx&\ge\displaystyle\int_{\R\cap \{u'\ne0\}}|u'(x)+\kappa u(x)|^pdx
\\
&\displaystyle
=\int_{0}^{\infty}\left\{\int_{\{u=t\}}\frac{|u'+\kappa t|^p}{|u'|} d{\mathcal H}^0\right\}dt
=\sum_{i=0}^N\int_{a_i}^{a_{i+1}}\left(\sum_{j=1}^{2M_j}\frac{|u'_{i,j}+\kappa t|^p}{|u'_{i,j}|} \right)dt\\
&\displaystyle= \sum_{i=0}^N\int_{a_i}^{a_{i+1}}\left(\sum_{j=1}^{M_j}\frac{\left||u'_{i,2j-1}|+\kappa t\right|^p}{|u'_{i,2j-1}|} 
+\frac{\left||u'_{i,2j}|-\kappa t\right|^p}{|u'_{i,2j}|} \right)dt\\
&\displaystyle\ge\sum_{i=0}^N\left[2^p(a_{i+1}-a_i)\sum_{j=1}^{M_j} \left( \frac {|u'_{i,2j-1}u_{i,2j}'|}{|u'_{i,2j-1}|+|u_{i,2j}'|}\right)^{p-1}\right]\\
&\displaystyle\ge\sum_{i=0}^N\left[2^p(a_{i+1}-a_i)M_j\left(\frac{1}{M_j}\sum_{j=1}^{M_j} \frac {|u'_{i,2j-1}|+|u_{i,2j}'|}{|u'_{i,2j-1}u_{i,2j}'|}\right)^{1-p}\right]\\
&\displaystyle\ge\sum_{i=0}^N\left[2^p(a_{i+1}-a_i)\left(\sum_{j=1}^{2M_j} |u'_{i,j}|^{-1}\right)^{1-p}\right]\\
&\displaystyle=\int_{0}^{\infty}\left\{\int_{\{u^\sharp=t\}}|u^\sharp{'}|^{p-1} d{\mathcal H}^0\right\}dt
=\int_\R|u^\sharp {'}(x)|^p dx
\end{align*}

Here we have used the fact that for $p \ge 2$
\begin{equation}\label{eq_convexity}
\frac{\left ||u'_{i,2j-1}|+\kappa t\right|^p}{|u'_{i,2j-1}|} 
+\frac{\left||u'_{i,2j}|-\kappa t\right|^p}{|u'_{i,2j}|}\ge2^p\left( \frac {|u'_{i,2j-1}u_{i,2j}'|}{|u'_{i,2j-1}|+|u_{i,2j}'|}\right)^{p-1}.
\end{equation}
Infact, differentiating with respect to $\kappa$, the left hand side of \eqref{eq_convexity} achieves the minimum when $$\kappa t=\frac{|u'_{i,2j-1}|^{1/(p-1)}|u'_{i,2j}|-|u'_{i,2j-1}||u'_{i,2j}|^{1/(p-1)}}{|u'_{i,2j-1}|^{1/(p-1)}+|u'_{i,2j}|^{1/(p-1)}}$$ hence
$$
\frac{\left||u'_{i,2j-1}|+\kappa t\right|^p}{|u'_{i,2j-1}|} 
+\frac{\left||u'_{i,2j}|-\kappa t\right|^p}{|u'_{i,2j}|}\ge \frac {(|u'_{i,2j-1}|+|u_{i,2j}'|)^p}{\left(|u'_{i,2j-1}|^{1/(p-1)}+|u_{i,2j}'|^{1/(p-1)}\right)^{p-1}}.
$$
Thereafter inequality \eqref{eq_convexity} becomes a trivial consequence of the convexity of the power $p-1$ and of the arithmetic-geometric mean inequality.

\end{proof}

Arguing exactly as in the proof of Proposition \ref{pr_polya}, we can easily generalize the result and state the following Corollary.
\begin{corollary}
Let $u$ be in $W^{1,p}(\R)$ ($p\ge 2$) a nonnegative compactly supported function.  Given any measurable function $f$ 
$$\int_\R|u'(x)+f(u(x))|^pdx\ge\int_\R|u^\sharp {'}(x)|^p dx.$$
\end{corollary}

\begin{remark}
We emphasize that the previous inequality is by no means a straightforward consequence of the classical P\'olya-Szeg\"o principle.
Indeed the reader may wonder whether
\begin{equation}\label{eq_nontrue}
\int_\R|u'(x)+f(u(x))|^pdx\ge\int_\R|u^\sharp {'}(x)+f(u^\sharp(x))|^p dx,
\end{equation}
since it is trivially true when $p=2$.
In fact the last inequality would imply Proposition \ref{pr_polya} since, using the symmetry of $u^\sharp$
$$\int_\R|u^\sharp {'}(x)+f(u^\sharp(x))|^p dx=\int_{\R^+}|u^\sharp {'}(x)+f(u^\sharp(x))|^p+ |u^\sharp {'}(x)-f(u^\sharp(x))|^pdx\ge\int_\R|u^\sharp {'}(x)|^p.$$
However, even if \eqref{eq_nontrue} holds true for $p=2$, it fails for $p>2$.

For instance we fix $f$ identically equal to one and $p>2$. For any $0<\varepsilon<1$ we can define a continuous piecewise linear function $u$, which is compactly supported in $[-1/(1-\varepsilon),1]$ such that
$$u'(x)=\left[\begin{array}{ll}
1-\varepsilon&x\in[-1/(1-\varepsilon),0]\\\\
-1&x\in[0,1]
\end{array}\right.$$
A simple calculation yields
$$\displaystyle \int_\R|u'(x)+f(u(x))|^pdx=2^p(1-\varepsilon(p/2-1))+o(\varepsilon)$$
and
$$\int_\R|u^\sharp {'}(x)+f(u^\sharp(x))|^p dx=2^p(1-(\varepsilon/2)(p/2-1))+o(\varepsilon).$$
Therefore for some $\varepsilon$ small enough \eqref{eq_nontrue} is false.

\end{remark}

\section{Weighted Wirtinger inequality (proof of Proposition \ref{pr_wirtinger})}\label{sec_wirtinger}

This section is devoted to the proof of Proposition \ref{pr_wirtinger}. For the reader convenience we have split the claim in three Lemmata. 
\begin{lemma}\label{lem_kappa}
Let $f$ be a nonnegative log-concave function defined on $[0,L]$ and $p>1$.
Then there exists $\kappa\in\R$ such that 
\begin{equation}\label{eq_kappa}
\mathop{\inf_{u\in W^{1,p}(0,L)}}_{\int_0^L f|u|^{p-2}u=0} \frac{\displaystyle \int_0^L f(x)|u'(x)|^p dx}{\displaystyle\int_0^L f(x)|u(x)|^p dx}\ge \mathop{\min_{u\in W^{1,p}(0,L)}}_{\int_0^L e^{\kappa x}|u|^{p-2}u=0} \frac{\displaystyle \int_0^L e^{\kappa x}|u'(x)|^p dx}{\displaystyle\int_0^L e^{\kappa x}|u(x)|^p dx}
\end{equation}
\end{lemma}
\begin{proof}
We assume that the function $f$ is smooth and bounded away from $0$ in $[0,L]$. The general case can be deduced by approximation arguments.\\
Under these assumptions the infimum on the left had side of \eqref{eq_kappa} is achieved by a function $u_\lambda$
belonging to $\left\{u\in W^{1,p}(0,L),\int_0^L f|u|^{p-2}u=0\right\}$, which is a solution to the following Neumann eigenvalue problem 
\begin{equation*}
\left\{\begin{array}{ll}
(-u'|u'|^{p-2})'=\lambda u|u|^{p-2}+h'(x)u'|u'|^{p-2} & x\in(0,L)\\\\
u'(0)=u'(L)=0.
\end{array}\right.
\end{equation*} 
Here $h(x)=\log f(x)$ is a smooth bounded concave function and $\lambda$ is the minimum of $\mathcal F$.

Standard arguments ensure that $u_\lambda$ vanishes in one and only one point namely $x_\lambda\in(0,L)$  and without loss of generality we may assume that ${u_\lambda}(L)<0<{u_\lambda}(0)$.

We claim that if $\kappa =h'(x_\lambda)$ then 

$$\lambda\ge \mathop{\min_{u\in W^{1,p}(0,L)}}_{\int_0^L e^{\kappa x}|u|^{p-2}u=0} \frac{\displaystyle \int_0^L e^{\kappa x}|u'(x)|^p dx}{\displaystyle\int_0^L e^{\kappa x}|u(x)|^p dx} \equiv\bar\lambda.$$

Arguing by contradiction we assume that $\lambda< \bar \lambda$. Therefore there exists a function $u_{\bar \lambda}$ solution to
\begin{equation*}
\left\{\begin{array}{ll}
(-u'|u'|^{p-2})'=\bar\lambda  u|u|^{p-2}+h'(x_\lambda) u'|u'|^{p-2} & x\in(0,L)\\\\
 u'(0)= u'(L)=0
\end{array}\right.
\end{equation*} 

Standard arguments ensures that $u_{\bar\lambda}$ vanishes in one and only one point namely $x_{\bar \lambda}\in(0,L)$, and we may always assume that $u_{\bar\lambda}(L) < 0 < u_{\bar \lambda}(0)$. Since $h'$ is non increasing in $[0,L]$, a strightforward consequence of the comparison principle applied to $u_\lambda$ and $u_{\bar\lambda}$ on the interval $[0,x_\lambda]$ enforces $x_{\bar \lambda}< x_\lambda$. On the other hand the comparison principle applied to $u_\lambda$ and $u_{\bar\lambda}$ on the interval $[x_\lambda,L]$ enforces $x_{\bar \lambda}> x_\lambda$ and eventually a contradiction arises.
\end{proof}

In view of Lemma \ref{lem_kappa} we restrict our investigation to (log-concave functions) $f$ of exponetial type (i.e. $f(x)=e^{\kappa x}$ for some $\kappa \in \R$).
For such a class of functional we are able to prove that the first nontrivial Neumann eigenvalue equals the first Dirichlet eigenvalue, namely we prove

\begin{lemma}\label{lem_allkappa}
For all $\kappa\in\R$ and $p>1$
\begin{equation}\label{eq_ND}
\mathop{\min_{u\in W^{1,p}(0,L)}}_{\int_0^L e^{\kappa x}|u|^{p-2}u=0} \frac{\displaystyle \int_0^L e^{\kappa x}|u'(x)|^p dx}{\displaystyle\int_0^L e^{\kappa x}|u(x)|^p dx}=\min_{u\in W_0^{1,p}(0,L)} \frac{\displaystyle\int_0^L e^{\kappa x}|u'(x)|^p dx}{\displaystyle\int_0^L e^{\kappa x}|u(x)|^p dx}
\end{equation}
\end{lemma}
\begin{proof}
If $v$ minimizes the left hand side of \eqref{eq_ND} then it solves
\begin{equation*}
\left\{\begin{array}{ll}
(-u'|u'|^{p-2})'=\mu u|u|^{p-2}+\kappa u'|u'|^{p-2} & x\in(0,L)\\\\
u'(0)=u'(L)=0,
\end{array}\right.
\end{equation*} 
where $$\mu=\mathop{\min_{u\in W^{1,p}(0,L)}}_{\int_0^L e^{\kappa x}|u|^{p-2}u=0} \frac{\displaystyle \int_0^L e^{\kappa x}|u'(x)|^p dx}{\displaystyle\int_0^L e^{\kappa x}|u(x)|^p dx}.$$
We denote by $x_0\in[0,L]$ the unique zero of $v$ and by
$$w(x)=\left[\begin{array}{ll}
\displaystyle
\left|\frac{v(x+x_0)}{v(L)}\right| & x\in [0,L-x_0]\\\\
\left|\dfrac{v(x-L-x_0)}{v(0)}\right| & x\in [L-x_0,L].
\end{array}\right.$$ 

We observe that any minimizer of the right hand side of \eqref{eq_ND} is a solution to
\begin{equation*}
\left\{\begin{array}{ll}
(-u'|u'|^{p-2})'=\nu u|u|^{p-2}+\kappa u'|u'|^{p-2} & x\in(0,L)\\\\
u(0)=u(L)=0,
\end{array}\right.
\end{equation*} 
where $$\nu=\min_{u\in W_0^{1,p}(0,L)} \frac{\displaystyle\int_0^L e^{\kappa x}|u'(x)|^p dx}{\displaystyle\int_0^L e^{\kappa x}|u(x)|^p dx}.$$

It is trivial to check that $w$ is a solution of 
\begin{equation*}
\left\{\begin{array}{ll}
(-u'|u'|^{p-2})'=\mu u|u|^{p-2}+\kappa u'|u'|^{p-2} & x\in(0,L)\\\\
u(0)=u(L)=0.
\end{array}\right.
\end{equation*} 

Hence $\mu\ge\nu$.
On the other hand $w$ is an eigenfunction having constant sign, and we can easily deduce by standard argument (for instance by using the maximum principle) that $\mu\le\nu$. Therefore we have $\mu=\nu$.
\end{proof}

\begin{lemma}\label{lem_fund}
Let $p\ge2$ and let $u$ be in $W_0^{1,p}(0,L)$.  Then there exists a nonnegative function $\tilde u \in W_0^{1,p}(0,L)$ such that 
\begin{align*}
&\int_0^L e^{\kappa x}|u'(x)|^p dx \ge \int_0^L |\tilde u'(x)|^p dx \\
&\int_0^L e^{\kappa x}|u(x)|^p dx = \int_0^L |\tilde u(x)|^p dx
\end{align*}
for all $\kappa\in\R$.
\end{lemma}

\begin{proof}
Setting  $v(x)= u(x) e^{\frac{\kappa}{p} x}$ we have
$$
\int_0^L e^{\kappa x} |u'(x)|^p dx = \int_0^L \left|v(x)'-\frac{k}{p} v(x)\right |^p dx
$$
and
$$
\int_0^L e^{\kappa x} |u(x)|^p dx= \int_0^L |v(x)|^p dx.
$$
The claim can be easily deduced from Proposition \ref{pr_polya} once we chose $\tilde u(x+L/2)=v^\sharp(x)\equiv\left(e^{\frac{\kappa}{p} x}|u(x)|\right)^\sharp$.
\end{proof}

\section{Reduction to one dimensional case (proof of Theorem \ref{teo_main})}\label{sec_red}
The aim of this section is to prove that Theorem \ref{teo_main} can be deduced from Proposition \ref{pr_wirtinger}. As we already mentioned the idea is based on a slicing metod worked out in \cite{pawe} and proved in a slightly different way also in \cite{be,acdur}. We outline the technique for the seek of completeness.

\begin{lemma}
\label{division}
Let $\Omega$ be a convex set in $\R^n$ having diameter $d$, and let $u$ be any function such that $\int_{\Omega_i}|u(x)|^{p-2}u(x)dx=0$. Then, for all positive $\varepsilon$, there exists a decomposition of the set $\Omega$ in mutually disjoint convex sets $\Omega_i$ ($i=1,...,k$) such that
$$\bigcup_{i=1}^k\bar\Omega_i=\bar\Omega$$
$$\int_{\Omega_i}|u(x)|^{p-2}u(x)dx=0$$
and for each $i$ there exists a rectangular system of coordinates such that
$$\Omega_i\subseteq\{(x_1,...,x_n)\in\R^n: 0\le x_1\le d_i, |x_\ell|\le\varepsilon, \ell=2,...,n \}\quad (d_i\le d, i=1,...,k)$$
\end{lemma}
\begin{proof}
Among all the $n-1$ hyperplanes of the form $a x_1+ b x_2 = c$, orthogonal to the 2-plane $\Pi_{1,2}$ generated by the directions $x_1$ and $x_2$, by continuity there exists certainly one that divides $\Omega$ into two nonempty subsets on each of which the integral of $u|u|^{p-2}$ is zero and their projections on $\Pi_{1,2}$ have the same area. We go on subdividing recursively in the same way both subset and eventually we stop when all the subdomains $\Omega^{(1)}_j$ ($j=1,...,2^{N_1}$) have projections with area smaller then $\varepsilon^2/2.$ Since the width $w$ of a planar set of area $A$ is bounded by the trivial inequality $w\le\sqrt{2A}$, each subdomain $\Omega^{(1)}_j$ can be bounded by two parallel $n-1$ hyperplanes of the form $a x_1+ b x_2 = c$ whose distance is less than $\varepsilon$. If $n=2$ the proof is completed, provided that we understand $\Pi_{1,2}$ as $\R^2$, the projection of $\Omega$ on $\Pi_{1,2}$ as $\Omega$ itself, and the $n-1$ orthogonal hyperplanes as lines.
If $n>2$ for any given $\Omega^{(1)}_j$ we can consider a rectangular system of coordinates such that the normal to the above $n-1$ hyperplanes which bound the set, points in the direction $x_n$. Then we can repeat the previous arguments and subdivide the set $\Omega^{(1)}_i$ in subsets $\Omega^{(2)}_j$ ($j=1,...,2^{N_2}$) on each of which the integral of $u|u|^{p-2}$ is zero and their projections on $\Pi_{1,2}$ have the same area which is less then $\varepsilon^2/2$. Therefore, any given $\Omega^{(2)}_j$, can be bounded by two $n-1$ hyperplanes of the form $a x_1+ b x_2 = c$ whose distance is less than $\varepsilon$. If $n=3$ the proof is over. If $n>3$ we can go on considering $\Omega^{(2)}_j$ and rotating the coordinate system such that the normal to the above $n-1$ hyperplanes which bound $\Omega^{(2)}_j$, points in the direction $x_{n-1}$ and such that the rotation keeps the $x_n$ direction unchanged. The procedure ends after $n-1$ iterations, at that point we have perfomed $n-1$ rotations of the coordinate system and all the directions have been fixed. Up to a translation, in the resulting coordinate system $$\Omega^{(n-1)}_j\subseteq\{(x_1,...,x_n)\in\R^n: 0\le x_1\le d_i, |x_\ell|\le\varepsilon, \ell=2,...,n \}$$
\end{proof}

\begin{proof}[Proof of Theorem \ref{teo_main}]
From \eqref{eq_quoz}, using the density of smooth functions in Sobolev spaces it will be enough to prove that 
$$ \dfrac{\displaystyle\int_\Omega |Du|^p}{\displaystyle\int_\Omega |u|^p}
\ge \left(\frac{\pi_p}{d}\right)^p $$
when $u$ is a smooth function with bounded second derivatives and $\int_{\Omega}|u(x)|^{p-2}u(x)dx=0$.

Let $u$ be any such function. According to Lemma \ref{division} we fix $\varepsilon$ and we decompose the set $\Omega$ in convex domains $\Omega_i$ ($i=1,...,k$). We use the notation of Lemma \ref{division} and we focus on one of the subdomains $\Omega_i$ and 
fix the reference system such that 
$$\Omega_i\subseteq\{(x_1,...,x_n)\in\R^n: 0\le x_1\le d_i, |x_\ell|\le\varepsilon, \ell=2,...,n \}.$$
For $t\in[0,D_i]$ we denote by $f(t)$ the $n-1$ volume of the intersection of $\Omega_i$ with the $n-1$ hyperplane $x_1=t$. Since $\Omega_i$  is convex, then by Brunn-Minkowski inequality (see \cite{gar}) $f$ is a log concave function in $[0,d_i]$.

Therefore, if for any $t\in[0,d_i]$ we denote by $v(t)=u(t,0,...,0)$, and by $C$ any constant depending just on $||u||_{C^{2}}$,  we have

\begin{equation}
\label{uno}
\left | \int_{\Omega_{i}} \left|\frac{\partial u}{\partial {x_1}} \right|^p \> dx-\int_0^{d_i} f(t) |v'(t)|^p \> dt\right| \leq C |\Omega_i| \varepsilon,
\end{equation}

\begin{equation}
\label{due}
\left | \int_{\Omega_{i}} \left|u \right|^p \> dx-\int_0^{d_i} f(t) |v(t)|^p \> dt\right| \leq C |\Omega_i| \varepsilon
\end{equation}

and

\begin{equation}
\label{tre}
\left | \int_0^{d_i} f(t) |v(t)|^{p-2}v(t) \> dt\right| \leq C |\Omega_i| \varepsilon.
\end{equation}

Since $d_i\le d$, applying Proposition \ref{pr_wirtinger} we have 

$$
\int_{\Omega_i} | D u|^p \>dx\ge  \int_{\Omega_{i}} \left|\frac{\partial u}{\partial {x_1}} \right|^p \> dx \ge
\left(\frac{\pi_p}{d}\right)^p \int_{\Omega_{i}} \left|u(x) \right|^p \> dx + C \varepsilon  |\Omega_i|.
$$

We sum up the last inequality over all the subdomains ($i=1,...,k$)

$$
\int_{\Omega} | D u|^p \>dx\ge \left(\frac{\pi_p}{d}\right)^p \int_{\Omega} \left|u(x) \right|^p \> dx + C \varepsilon  |\Omega|.
$$

and we let $\varepsilon\to 0$ to obtain the desired inequality.

\end{proof}

\end{document}